\documentclass{degruyter-journal-a} 
\usepackage[english]{babel}

\newtheorem{theo}{Theorem}[section]

\newtheorem{tv}{Proposition}[section]
\newtheorem{n}{Corollary}[section]
\theoremstyle{remark}

\newcommand{\R}{\mathbb{R}}
\newcommand{\Expect}{\mathsf{E}}
\newcommand{\Real}{\operatorname{Re}}
\newcommand{\sign}{\operatorname{sign}}

\newcommand{\myiiint}{\int}

\title{Wave equation with a coloured stable noise}

\lastnameone{Pryhara}
\firstnameone{Larysa}
\nameshortone{L.~Pryhara}
\addressone{Taras Shevchenko National University of Kyiv, Mechanics and Mathematics Faculty, Volodymyrska 64, 01601 Kyiv}
\countryone{Ukraine}
\emailone{pruhara7@gmail.com}

\lastnametwo{Shevchenko}
\firstnametwo{Georgiy}
\nameshorttwo{G.~Shevchenko}
\addresstwo{Taras Shevchenko National University of Kyiv, Mechanics and Mathematics Faculty, Volodymyrska 64, 01601 Kyiv}
\countrytwo{Ukraine}
\emailtwo{zhora@univ.kiev.ua}

\abstract{We define a random measure generated by a real anisotropic harmonizable fractional stable field $Z^H$ with stability parameter $\alpha\in(1,2)$ and Hurst index $H\in(1/2,1)$ and prove that the measure is  $\sigma$-additive in probability.  An integral with respect to this measure is constructed, which enables us to consider a wave equation in $\R^3$ with a random source generated by $Z^H$. We show that the solution to this equation, given by Kirchhoff's formula, has a modification, which is H\"older continuous of any order up to $(3H-1)\wedge 1$. In the case where $H\in(2/3,1)$, we show further that the modification is absolutely continuous. 
}



\begin{document}

\maketitle

\section{Introduction}

Being a common tools to model complex systems with randomness, stochastic partial differential equations have attracted a considerable attention. The number of articles in this area is constantly growing, and we mention only few of them which consider subjects closely related to our work; further references are available therein. The sample properties of wave equations with Gaussian random noise were studied in \cite{dalang-frangos,dalang-sanz,millet-morien,quer-tindel,walsh}. Several authors investigated equation with square integrable L\'evy noise, where the same martingale methods are applicable, see \cite{applebaum,peszat-zabczyk} and references therein. 

However, in the heavy-tailed situation, there are only few results. Articles \cite{chong,mytnik,mueller-mytnik-stan,pryhara-shevchenko1,yang-zhou} are devoted to heat equations with stable noise. To the best of our knowledge, there are no articles studying the wave equation with a heavy-tailed noise, neither there are attempts to consider equations with coloured stable noise.  It is worth to mention that some authors investigated stochastic partial differential equations with a general stochastic measure, where, in particular, no assumptions on the integrability of measure are made. For example, the heat equation with a general stochastic measure was considered in \cite{bodnarchuk-shevchenko,radch2}, the wave equation, in \cite{bodnarchuk,gorodnya}. However, in these works only few results are proved and in rather restricted cases; it is probably the generality of exposition what leads to such restrictions. 

In this article we continue our research started in \cite{pryhara-shevchenko2}, where a planar wave equation was studied, and \cite{pryhara-shevchenko3}, which, as the present work, was devoted to wave equation in $\R^3$. The main object of our study is the wave equation
\begin{equation}\label{eq:wave-3d}
\left\{
\begin{aligned}
\Big(&\frac{\partial ^{2}}{\partial t^2}- a^2 \Delta\Big)U\left(x,t\right)=\dot Z^H(x), x\in\R^3, t>0,\\
&U(x,0)=0, \\
&\frac{\partial U}{\partial t}(x,0)= 0,
\end{aligned}
\right.
\end{equation}
where the source is a spatial random noise with symmetric $\alpha$-stable distribution: it is a derivative $\dot Z^H(x)$ of a real anisotropic harmonizable fractional stable field $Z^H$. Thus, the random noise is ``coloured'' in the sense that its increments are not independent. To give a meaning to this equation requires defining a random measure generated by $Z^H$ and an integral with respect to this measure. As far as we know, such questions did not appear in the previous literature. 

The rest of article is organized as follows. Section 2 contains preliminary information on stable random variables, measures and integrals. In Section 3, we recall the definition and properties of real anisotropic harmonizable fractional stable field, construct a random measure generated by this field, define integrals with respect to this measure and study their properties. In Section 4 we prove results on the sample path properties of the solution to \eqref{eq:wave-3d}. 

\section{Preliminaries}
Throughout the paper, $C$ will denote a generic constant; its value may vary between lines. Random constants will be denoted by $C(\omega)$.

In this article we will consider symmetric $\alpha$-stable $(S \alpha S)$ random variables. In this section we  will provide essential information about them; additional detail may be found in  \cite{samor-taqqu}. 

For $\alpha \in(0,2)$, a random variable $\xi$ is called $S\alpha S$ with scale parameter $||\xi||_\alpha$ if its characteristic function is 
$$\Expect \left[e^{i\lambda\xi}\right]=e^{-\lambda||\xi||_\alpha^{\alpha}}, \lambda \in \R.$$

A crucial role in the construction of processes and fields with stable distribution is played by independently scattered $S\alpha S$ random measure; in this article it is enough to consider a measure on $\R^3$. This is a function $M: \mathcal B_f(\R^3)\times\Omega\rightarrow\R$, where $\mathcal B_f(\R^3)$ is a family of Borel sets of finite Lebesgue measure, having the following properties:
\begin{enumerate}
	\item for any $A \in\mathcal B_f(\R^3)$, the random variable $ M (A)$ is  $S\alpha S$  with scale parameter equal to $\lambda(A)$ the Lebesgue measure of $A$;
	\item for any disjoint sets $A_1, \dots ,A_n \in\mathcal B_f(\R^3)$, the values $M(A_1),\dots ,M(A_n)$  are independent;
	\item for any disjoint sets $A_1, A_2,\dots  \in\mathcal B_f(\R^3)$ such that $\bigcup_{n=1}^\infty A_n\in\mathcal B_f(\R^3)$, the series $\sum_{n=1}^\infty M(A_n)$ converges almost surely and $M\left(\bigcup_{n=1}^\infty A_n\right)=\sum_{n=1}^\infty M(A_n)$ almost surely.
\end{enumerate}
For a function $f \in L^\alpha(\R^3)$, the integral 
$$I(f)=\myiiint_{\R^3}f(x)M(dx)$$
is defined as a limit in probability of integrals of finitely supported simple functions, and there is an isometric property:
$$\left\|I(f)\right\|^\alpha_\alpha=\myiiint_{\R^3}\left|f(x,t)\right|^\alpha dx.$$

A convenient tool to study stable random variables is the LePage series representation, defined for the measure $M$ as follows. Let $\varphi$ be any positive continuous probability density function on $\R^3$, and $\left\{\Gamma_{k},k\geq1\right\}$, $\left\{\xi_{k},k\geq1\right\}$, $\left\{g_k,k\geq1\right\}$ be independent families of random variables satisfying
\begin{itemize}
	\item $\left\{\Gamma_{k},k\geq1\right\}$ is a sequence of arrivals of a Poisson process with unit intensity;
	\item $\left\{\xi_{k},k\geq1\right\}$ are iid vectors in $\R^3$ with density $\varphi$;
	\item $\left\{g_k,k\geq1\right\}$ are iid centered Gaussian variables with  $\Expect [\left|g_k\right|^\alpha]=1.$
\end{itemize}
Then $\{M(A),A\in\mathcal{B}_f(\R^3)\}$ has the same distribution as 
\begin{equation}\label{eqLP}
M'(A)=C_\alpha \sum_{k\geq1}\Gamma_k^{-{1}/{\alpha}}\varphi(\xi_k)^{-1/\alpha} \mathbf{1}_A(\xi_k)g_k,\ A\in \mathcal B_f(\R^3),
\end{equation}
where $C_\alpha=\big(\frac{\Gamma(2-\alpha)\cos \frac{\pi\alpha}{2}}{1-\alpha}\big)^{1/\alpha}$; the series converges  almost surely for any $A\in\mathcal B_f(\R^3)$. 
There is also the LePage series representation for integrals: for any $f_1,f_2,\dots, f_n \in L^\alpha(\R^3)$ the vector $\left(I(f_1), I(f_2), \dots , I(f_n)\right)$ has the same distribution as $\left(I'(f_1), I'(f_2), \dots , I'(f_n)\right)$, where
\begin{equation}\label{eqLPI}
I'(f)=C_\alpha\sum_{k\geq1}\Gamma_k^{-{1}/{\alpha}}\varphi(\xi_k)^{-1/\alpha} f(\xi_k)g_k.
\end{equation}
 
Further we will assume without loss of generality that $M$ is given by  \eqref{eqLP} with 
$$\varphi(x)=\prod_{l=1}^3\frac{K}{\left|x_l\right| \left(\bigl|\log\left|x_l\right| \bigr|+1\right)^{1+\eta}}, $$
where $\eta>0$ is a fixed number, and
$K=\left(\int_{-\infty}^{+\infty}\left|x\right|^{-1}\left(\bigl|\log\left|x\right| \bigr|+1\right)^{-1-\eta}dx\right)^{-1}$ is the normalizing constant.
Respectively, the integral $I(f)=\myiiint_{\R^3}f(x)M(dx)$ is given by \eqref{eqLPI}.

We will need several auxiliary results. For convenience, we reformulate them to suit our needs.
\begin{tv}[\protect{\cite[Proposition 5.3.1]{samor-taqqu}}]\label{l31}
Let $X_n = I(f_n)$ with some $f_n\in L^\alpha(\R^3)$, $n\ge 1$. Then 
$$
I(f_n)\overset{\mathsf{P}}{\longrightarrow} I(f),\ n\to\infty,
$$
if and only if $\myiiint_{\R^3} |f_n(x) - f(x)|^\alpha dx\to 0$, $n\to\infty$. 
\end{tv}

\begin{theo}\label{tv32t41}
Assume that $\alpha \in (1,2)$. For  $\mathbf{T}\in\mathcal{B} (\R^m)$ and measurable  $f_t(x): \mathbf{T}\times \R^3\rightarrow \R$ 
define
$$X(t)=\myiiint_{\R^3}f_t(x)M(dx), \, t\in \mathbf{T}.$$
If
\begin{equation}\label{i1}
\int_{\mathbf{T}}\left|X(t)\right|dt<\infty
\end{equation}
almost surely, then
\begin{equation*}
\int_{\mathbf{T}}X(t)dt=\myiiint_{\R^3}\int_{\mathbf{T}}f_t(x)dt\, M(dx)
\end{equation*}
almost surely.
A sufficient condition for \eqref{i1} to hold is
$$\int_{\mathbf{T}}\left(\myiiint_{\R^3}\left|f_t(x)\right|^\alpha dx\right)^{1/\alpha}dt<\infty.$$
\end{theo}
\begin{proof}
This is a compilation of Theorem 3.3 and Theorem 4.1 from \cite{samor}. The measurability of $X$, assumed in these theorems, in our case is a consequence of the series representation \ref{eqLPI}.
\end{proof}

\section{Real anisotropic harmonizable fractional stable field}

The real anisotropic harmonizable fractional stable field with Hurst parameter $H\in(0,1)$ is defined as
$$Z^H(x)=\Real\myiiint_{\R^3}\prod_{l=1}^{3}\frac{e^{ix_ly_l}-1}{\left|y_l\right|^{H+1/\alpha}}M(dy)$$
for technical reasons we will restrict ourselves to the case $\alpha\in(1,2)$, $H\in (1/2,1)$.
The properties of $Z^H$ were established in \cite{kono-maejima}, in \cite{rhmsf} its multifractional was defined and investigated. In particular, it was proved in the cited articles that this field has a modification, which is locally H\"older continuous of any order $\beta\in (0,H)$.
Thanks to our assumptions about $M$, $Z^H$ can be represented as its LePage series: 
$$Z^H(x)=C_\alpha\Real\sum_{k=1}^{\infty}\Gamma_k^{-1/\alpha}\prod_{l=1}^3\frac{e^{ix_l\xi_{k,l}}-1}{\left|\xi_{k,l}\right|^{H+1/\alpha}}\varphi\left(\xi_k\right)^{-1/\alpha}g_k,$$
де 
$$\varphi(x)=\prod_{l=1}^3\frac{K}{\left|x_l\right| \left(\bigl|\log\left|x_l\right| \bigr|+1\right)^{1+\eta}}. $$

Our aim is to define a random measure corresponding to $Z^H(x)$. Towards this end, we will use the following heuristic reasoning. Differentiate  $Z^H(x)$ informally:
\begin{gather*}
\frac{\partial^3}{\partial x_1\partial x_2 \partial x_3}Z^H(x)=\Real\left(-i\myiiint_{\R^3}e^{i(x,y)}\prod_{l=1}^3\frac{\sign y_l}{\left|y_l\right|^{H+1/\alpha-1}}M(dy)\right)\\
=\myiiint_{\R^3}\sin(x,y)\prod_{l=1}^3\frac{\sign y_l}{\left|y_l\right|^{H+1/\alpha-1}}M(dy).
\end{gather*}

Note that the integral is not well defined in general, but it allows to define the random measure corresponding to $Z^H$ as a result of formal differentiation and change of order of integration:
\begin{gather*}
Z^H(A)=\myiiint_{A}\frac{\partial^3}{\partial x_1\partial x_2 \partial x_3}Z^{H}(x)\,dx\\
=\myiiint_{A}\myiiint_{\R^3}\sin(x,y)\prod_{l=1}^3\frac{\sign y_l }{\left|y_l\right|^{H+1/\alpha-1}}M(dy)\, dx\\
=\myiiint_{\R^3}\prod_{l=1}^3\frac{\sign y_l }{\left|y_l\right|^{H+1/\alpha-1}}\myiiint_{A} \sin(x,y) dy\, M(dz).
\end{gather*}
In other words, we set
\begin{equation}\label{zh}
\begin{gathered}
Z^H(A)=\myiiint_{\R^3}\myiiint_A\sin(x,y)dx\prod_{l=1}^3\frac{\sign y_l }{\left|y_l\right|^{H+1/\alpha-1}}M(dy)
\end{gathered}
\end{equation}
by definition.

\begin{theo}\label{1t} 
For any $A\in B_f\left(\R^3\right)$, the integral in \eqref{zh} is well defined.
\end{theo} 
\begin{proof}
Denote
$$f_A(y)=\myiiint_A\sin(x,y)dx\prod_{l=1}^3\frac{\sign y_l }{\left|y_l\right|^{H+1/\alpha-1}}.$$
We need to show that 
$$\myiiint_{\R^3}\left|f_A(y)\right|^\alpha dy <\infty.$$
Note that $f_A(y)=A^H\mathbf{1}_A(y)$, where $A^H$ is defined in \eqref{ah}. Since $\mathbf{1}_A(y) \in L^1(\R^3)\cap L^2(\R^3)$, the finiteness of integral follows from Proposition~\ref{a1}. 
\end{proof}

In \cite{pryhara-shevchenko3} we have proposed to define the integral with respect to $Z^H$ as follows. For $\varepsilon>0$ define smooth approximations of $Z^H$ by
\begin{gather*}
Z^{H,\varepsilon}(x)=\myiiint_{\R^3}\sin(x,y)e^{-\varepsilon^2\left|y\right|^2/2}\prod_{l=1}^3\frac{\sign y_{k,l}}{\left|y_{k,l}\right|^{H+1/\alpha-1}}M(dy).
\end{gather*}
It was proved in \cite{pryhara-shevchenko3} that there exists a weak derivative $\frac{\partial^3}{\partial x_1\partial x_2 \partial x_3}Z^{H, \varepsilon}(x)$ and, in view of this,  the following definition was proposed:
\begin{equation}\label{lim2}
\myiiint_{\R^3}f(x)Z^H(dx)=\lim_{\varepsilon\rightarrow 0+}\myiiint_{\R^3}f(x)\frac{\partial^3}{\partial x_1\partial x_2 \partial x_3}Z^{H,\varepsilon}(x)\,dx,
\end{equation}
provided that the limit in probability exists. Let us show that \eqref{zh} agrees with this definition.
\begin{theo}\label{iz}
For any $A\in\mathcal B_f(\R^3)$, the following convergence in probability holds:
$$Z^H(A)=\lim_{\varepsilon\rightarrow 0+}\myiiint_{A}Z^{H,\varepsilon}(x)dx.$$
\end{theo}
\begin{proof}
Denote
$$f_{A,\varepsilon}(y)= e^{-\varepsilon^2|y|^2/2} f_A(y)= e^{-\varepsilon^2|y|^2/2}\prod_{l=1}^3\frac{1}{\left|y_l\right|^{H+1/\alpha-1}} \myiiint_A\sin(x,y)dx$$
and
\begin{equation}\label{zhe}
\begin{gathered}
Z^{H,\varepsilon}(A)=\myiiint_{\R^3}\myiiint_A\sin(x,y)e^{-\varepsilon^2|y|^2/2}\prod_{l=1}^3\frac{1}{\left|y_l\right|^{H+1/\alpha-1}}dx\,M(dy)\\
=\myiiint_{\R^3}f_{A,\varepsilon}(y)M(dy).
\end{gathered}
\end{equation}
Let us first prove the possibility to change the order of integration in  \eqref{zhe}. To this end, according to Theorem \ref{tv32t41}, it is enough to show that
$$\myiiint_A\left(\myiiint_{\R^3}\left|\sin(x,y)\right|^\alpha e^{-\alpha\varepsilon^2\left|y\right|^2/2}\prod_{l=1}^3\frac{1}{\left|y_l\right|^{\alpha(H-1)+1}}dy\right)^{1/\alpha}dx<\infty.$$
Estimate
\begin{gather*}
\myiiint_{\R^3}\left|\sin(x,y)\right|^\alpha e^{-\alpha\varepsilon^2\left|y\right|^2/2}\prod_{l=1}^3\frac{dy}{\left|y_l\right|^{\alpha(H-1)+1}}\\
\leq \myiiint_{\R^3}\left(1\wedge\left|x\right|\left|y\right|\right)^\alpha e^{-\alpha\varepsilon^2\left|y\right|^2/2}\prod_{l=1}^3\frac{dy}{\left|y_l\right|^{\alpha(H-1)+1}}:=I.
\end{gather*}
Through the spherical change of variables 
\begin{equation}\label{sfz}
\begin{gathered}
 y_1=\rho\sin\theta\cos\nu; y_2=\rho\sin\theta\sin\nu; y_3=\rho\cos\theta; \rho>0, \theta\in [0,\pi], \nu\in [0,2\pi], 
\end{gathered}
\end{equation}
we get
\begin{gather*}
I=\int_0^\infty\int_0^\pi\int_0^{2\pi}\left(1\wedge\left|x\right|\rho\right)^\alpha \frac{e^{-\alpha\varepsilon^2\rho^2/2}\rho^2\left(\sin\theta\right)^{2\alpha(H-1)-1}d\rho\, d\theta \,d\varphi}{\rho^{3\alpha(H-1)+3}\left(\cos\varphi\sin\varphi\cos\theta\right)^{1-\alpha(H-1)}}\\
\leq C\int_0^\infty \left(1\wedge\left|x\right|\rho\right)^\alpha e^{-\alpha\varepsilon^2\rho^2/2}\rho^{3\alpha(1-H)-1}d\rho<\infty.
\end{gather*}
Therefore, 
$$\myiiint_A \frac{\partial^3}{\partial x_1\partial x_2 \partial x_3}Z^{H,\varepsilon}(x)dx=\myiiint_{\R^3}f_{A,\varepsilon}(y)M(dy).$$
Now 
\begin{gather*}
\myiiint_{\R^3}\left|f_A(y)-f_{A,\varepsilon}(y)\right|^\alpha dy
\leq\myiiint_{\R^3}\big|e^{-\alpha\varepsilon^2\left|y\right|^2/2}-1\big|\left|f_A(y)\right|^\alpha dy\rightarrow 0, \varepsilon\rightarrow 0,
\end{gather*}
in view of the dominated convergence theorem, since $\int_{\R^3} \left|f_A(y)\right|^\alpha dy<\infty$ by Theorem~\ref{1t}.
Therefore, by Proposition~\ref{l31}, 
\begin{gather*}
\myiiint_A \frac{\partial^3}{\partial x_1\partial x_2 \partial x_3}Z^{H,\varepsilon}(x)dx \overset{\mathsf{P}}{\longrightarrow}Z^H(A), n\rightarrow \infty.\qedhere
\end{gather*}
\end{proof}

Now we will prove that $Z^H(A)$ is a stochastic measure, that is, that is a $\sigma$-additive in probability function of $A\in\mathcal{B}_f(\R^3)$. From the definition and linearity of the integral w.r.t.\ $M$ it is clear that $Z^H(A)$ is additive. Therefore, it suffices to prove only continuity at zero, which is done in the following proposition.
\begin{tv}
Let $\left\{A_n, n\geq 1\right\}\subset B_f(\R^3)$ be such that for any $n \geq 1$, $A_{n+1}\subset A_n$, and $\cap_{n=1}^\infty A_n=\varnothing$. Then $Z^H(A_n)\overset{\mathsf{P}}{\longrightarrow} 0, n\rightarrow\infty.$
\end{tv} 
\begin{proof}
Thanks to Proposition~\ref{l31}, we need to show that
$$\myiiint_{\R^3}\left|f_{A_n}(y)\right|^\alpha dy\rightarrow 0, n\rightarrow\infty.$$
This follows immediately from Proposition \ref{a1} and the continuity of the Lebesgue measure, as
\begin{gather*}
||f_{A_n}||_{L^1(\R^3)} + ||f_{A_n}||_{L^2(\R^3)}\leq   \lambda(A_n) + \lambda(A_n)^{1/2}. \qedhere
\end{gather*}
\end{proof} 
\begin{n}
$\left\{Z^H(A), A \in B_f(\R^3)\right\}$ is a stochastic measure.
\end{n}

We turn now to integration with respect to the measure $Z^H$. Since its increments are dependent, the standard integration theory, as described in Section 2, is not available. Nevertheless, we can proceed in a standard way. For a simple function
$$g(x)=\sum_{k=1}^na_kI_{A_k}(x)$$
with $A_k\in B_f(\R^3)$, $k=1,\dots,n$, define
$$I^H(g)=\myiiint_{\R^3}g(x)Z^H(dx)=\sum_{k=1}^na_kZ^H(A).$$
Observe that in this case
\begin{gather*}
I^H(g)=\myiiint_{\R^3}\sum_{k=1}^na_k f_{A_k}(y)M(dy)=\myiiint_{\R^3} A^Hg(y)M(dy),
\end{gather*}
where $A^H$ is given by \eqref{ah}. Then it is natural to define for arbitrary $g$
\begin{equation}\label{ih}
I^H(g)=\myiiint_{\R^3}f(x)Z^H(dx)=\myiiint_{\R^3}A^Hg(y)M(dy).
\end{equation}
Thanks to Proposition~\ref{a1},  this is well defined for any $g \in L^1(\R^3) \cap L^2(\R^3)$. Moreover, the map $I^H$ is continuous in the following sense.
\begin{theo}
Let $\left\{g_n,n\geq 1\right\}\subset L^1(\R^3)\cap L^2(\R^3)$ be such that $g_n\rightarrow g$, $n\rightarrow \infty$, both in $L^1(\R^3)$ and in $L^2(\R^3)$. Then
$$I^H(g_n)\overset{\mathsf{P}}{\longrightarrow }I^H(g), n\rightarrow\infty.$$
\end{theo}
\begin{proof}
From linearity of operator $A^H$ and that of integral w.r.t.\ $M$ we have
$$I^H(g_n)-I^H(g)=I^H(g_n-g).$$
Therefore, it is enough to prove that
$$I^H(g_n-g)\rightarrow 0, n\overset{\mathsf{P}}{\longrightarrow } \infty,$$
which follows from Propositions~\ref{l31} and \ref{a1}.
\end{proof}
In particular, the integral $I^H(f) = \int_{\R^3} f(x)Z^H(dx)$ may be equivalently defined as a limit in probability of integrals of simple functions approximating $f$ in $L^1(\R^3)\cap L^2(\R^2)$. Another important observation is that our definition agrees with that given in \cite{pryhara-shevchenko3}. 
\begin{theo}\label{thm:agreementofdefs}
For any $f\in L^1(\R^3)\cap  L^2(\R^3)$, the convergence in probability \eqref{lim2} holds.
\end{theo}
\begin{proof}
The proof is the same as in Theorem~\ref{iz}.
\end{proof}

\section{Wave equation with coloured stable noise}
Let us return to the wave equation with coloured $S\alpha S$ noise:
\begin{equation}\label{eqR31}
\left\{
\begin{aligned}
\Big(&\frac{\partial ^{2}}{\partial t^2}- a^2 \Delta\Big)U\left(x,t\right)=\dot Z^H(x),\ x\in\R^3, t>0,\\
&U(x,0)=0, \\
&\frac{\partial U}{\partial t}(x,0)= 0.
\end{aligned}
\right.
\end{equation}
This equation was already studied in \cite{pryhara-shevchenko3}, wherein we proved that its candidate solution given by Kirchhoff's formula
\begin{equation*}
\begin{gathered}
U(x,t)=\frac{1}{4\pi a}\myiiint_{y:\left|x-y\right|<at}\frac{1}{\left|x-y\right|}Z^H(dy)
\end{gathered}
\end{equation*}
is a weak (generalized) solution, that is for any $\theta (x,t)\in C_{fin}^{\infty}(\R^3\times\R^{+})$ almost surely it holds that
\begin{equation*} 
\int_0^{\infty}\myiiint_{\R^3}U(x,t)\left(\frac{\partial^2}{\partial t^2}\theta (x,t)-a^2\Delta\theta (x,t)\right)dx\,dt=\int_0^{\infty}\myiiint_{\R^3}\theta (x,t)Z^H(dx)dt.
	\end{equation*}
In \cite{pryhara-shevchenko3}, the integral in \eqref{eqR31} was understood in the sense \eqref{lim2}. Nevertheless, thanks to Theorem~\ref{thm:agreementofdefs}, this agrees with the definition \eqref{ih} taken in the present paper. Indeed, defining
$$f_{t,x}(y)=\frac{1}{4\pi a\left|x-y\right|}I_{\left|x-y\right|<at}$$
so that
$$U(x,t)=\myiiint_{\R^3}f_{x,t}(y)Z^H(dy),$$
we have $f_{t,x}\in L^1(\R^3)\cap L^2(\R^3)$.
As a result, 
\begin{gather*}
U(x,t)=\myiiint_{\R^3}A^H f_{x,t}(y) M(dy)\\
=\frac{1}{a}\myiiint_{\R^3}\sin(x,y)\prod_{l=1}^{3}\frac{\sign y_l}{\left|y_l\right|^{H+1/\alpha-1}}\frac{1-\cos at\left|y\right|}{\left|y\right|^2}M(dy)
\end{gather*}
(the expression for $A_H f_{x,t}(y)$ is computed in \cite[p. 151]{pryhara-shevchenko3}).

\section{Sample properties of solution to wave equation with coloured stable noise}

Write the random field $U(x,t)$ as its LePage series: 
\begin{equation*}
U(x,t)=\frac{C_\alpha}{a}\sum_{k=1}^\infty\Gamma_k^{-1/\alpha}\varphi(\xi_k)^{-1/\alpha}\prod_{l=1}^3\frac{\sign \xi_{k,l}}{\left|\xi_{k,l}\right|^{H+1/\alpha-1}}\sin\left(x,\xi_k\right)\frac{1-\cos at\left|\xi_k\right|}{\left|\xi_k\right|^2}g_k.
\end{equation*}
For notational simplicity assume that the underlying probability space has the following structure:
$$\left(\Omega, \mathcal F, \mathsf{P}\right)=\left(\Omega_{\Gamma}\times \Omega_{\xi} \times \Omega_{g}, \mathcal F_{\Gamma} \otimes \mathcal F_{\xi} \otimes \mathcal F_{g}, \mathsf{P}_{\Gamma}\otimes \mathsf{P}_{\xi}\otimes \mathsf{P}_{g} \right),$$
and for all $\omega= \left(\omega_{\Gamma},\omega_{\xi},  \omega{g}\right)$, $k\geq 1 : \Gamma_{k}(\omega)=\Gamma_{k}(\omega_{\Gamma}), \xi_{k}(\omega)=\xi_{k}(\omega_{\xi}),$ $g_{k}(\omega)=g_{k}(\omega_{g})$.
\begin{theo}
1. The random field $U$ has a modification, which is $\gamma$-H\"older continuous in $t$ and locally in $X$ for any $\gamma \in (0,(3H-1)\wedge 1)$. Moreover,  for any $\delta>0$ this modification satisfies
\begin{gather*}
\sup_{\substack{
\left|x'\right|, \left|x''\right|\leq R,\
t',t'' \in [0,T] \\
\left|x'-x''\right|\le h,
\left|t'-t''\right|\le h}}\left|U(x',t')-U(x'',t'')\right|\leq C(\omega)h^{(3H-1)\wedge 1}\left|\log h\right|^{3/\alpha-1/2+\delta}
\end{gather*}
for all $R>0$ and all $h>0$ small enough.

2. If additionally $H \in \left(2/3, 1\right),$ then this modification is absolutely continuous in each variable.
\end{theo}
\begin{proof}
1. Estimate
\begin{gather*}
\Expect_g\left[\left|U(x',t')-U(x'',t'')\right|^2\right]\\
\leq 2 \left(\Expect_g\left[\left|U(x',t')-U(x'',t')\right|^2\right]+\Expect_g\left[\left|U(x'',t')-U(x'',t'')\right|^2\right]\right)
\end{gather*}
Let $h \in \left(0,{1}/{2}\right)$. Define
\begin{gather*}
a_1(h)=\sup_{\substack{
t \in [0,T], x',x''\in\R^3 \\
\left|x'-x''\right|\le h
}}\Expect_g\left[\left|U(x',t)-U(x'',t)\right|^2\right],\\
a_2(h)=\sup_{\substack{
\left|x\right|\leq R,\ t',t'' \in [0,T]\\ \left|t'-t''\right|\leq h 
}}\Expect_g\left[\left|U(x,t')-U(x,t'')\right|^2\right].
\end{gather*}
Using the LePage representation, estimate
\begin{gather*}
\Expect_g\left[\left|U(x',t)-U(x'',t)\right|^2\right]\\
\leq a^{-2} C_\alpha^2 \sum_{k=1}^\infty\Gamma_k^{-2/\alpha}\varphi(\xi_k)^{-2/\alpha}\frac{\left|1-\cos at\left|\xi_k\right|\right|^2}{\left|\xi_k\right|^4} \big|\sin (x',\xi_k) - \sin(x'',\xi_k)\big|^2\times\\
\times\prod_{l=1}^3\frac{1}{\left|\xi_{k,l}\right|^{2H+2/\alpha-2}}.
\end{gather*}
From simple inequalities
$$\frac{\left|1-\cos at\left|y\right|\right|}{\left|y\right|^2}\leq 1\wedge|y|^{-2},\ \big|\sin (x',\xi_k) - \sin(x'',\xi_k)\big|\le 2\wedge \big(|y|\cdot |x'-x''|\big)$$
we get the following estimate:
\begin{equation}\label{e}
\begin{gathered}
a_1(h)\leq a^{-2} C_\alpha^2 \sum_{k=1}^\infty\Gamma_k^{-2/\alpha}\varphi(\xi_k)^{-2/\alpha} \big(1\wedge\left|\xi_k\right|^{-4}\big)\big(4 \wedge (|\xi_k|^2 h^2)  \big)\\\times \prod_{l=1}^3\frac{1}{\left|\xi_{k,l}\right|^{2H+2/\alpha-2}}
:= a^{-2} C_\alpha^2\sum_{k=1}^\infty\Gamma_k^{-2/\alpha}  Q\left(h,\xi_k\right).
\end{gathered}
\end{equation}
Consider
\begin{gather*}
\Expect_\xi\left[Q\left(h,\xi_k\right)\right]=\myiiint_{\R^3}\varphi(y)^{1-2/\alpha} \big(1\wedge\left|y\right|^{-4}\big) \big(4\wedge (|y|^2 h^2)\big)\prod_{l=1}^3\frac{1}{\left|y_{l}\right|^{2H+2/\alpha-2}}dy\\
=K^{3-6/\alpha}\myiiint_{\R^3}\big(1\wedge\left|y\right|^{-4}\big) \big(4\wedge (|y|^2 h^2)\big)\prod_{l=1}^3\frac{dy}{\left|y_l\right|^{2H-1}\left(\left|\log\left|y_l\right|\right|+1\right)^{(1+\eta)(1-2/\alpha)}}. 
\end{gather*}
Now we make the spherical change of variables \eqref{sfz} and estimate the logarithms as
\begin{equation*} 
\begin{gathered}
\big(\big|\log \left|\rho a(\nu,\theta)\right|\big|+1\big)^d\leq \bigl(\left|\log \rho\right|+ \bigl|\log \left|a(\nu,\theta)\right|\bigr|+1\bigr)^d\\
\leq \left(\left|\log \rho\right|+1\right)^d\bigl(\bigl|\log \left|a(\nu,\theta)\right|\bigr|+1\bigr)^d,\\
\end{gathered}
\end{equation*}
where $d=(1+\eta)(2/\alpha-1)$, and $a(\nu,\theta)$ is one of the functions $\sin\theta\cos\nu$, $\sin\theta\sin\nu$, $\cos\theta$. Thus we get
\begin{gather*}
\Expect_{\xi}\left[Q(h,\xi_k)\right]\leq C \int_0^\infty\frac{\left(1\wedge \rho^{-4}\right)\left(\rho^2h^2\wedge 1\right)\left(\left|\log \rho\right|+1\right)^{3d}}{\rho^{6H-5}}\\ \times\int_0^{2\pi}\int_0^\pi\bigl(\bigl|\log\left|\sin\theta\cos\nu\right|\bigr|+1\bigr)^d\bigl(\bigl|\log\left|\sin\theta\sin\nu\right|\bigr|+1\bigr)^d
\\\times  \bigl(\bigl|\log\left|\cos\theta\right|\bigr|+1\bigr)^d\left|\sin\theta\right|d\theta \, d\nu \, d\rho \\
\leq C \int_0^\infty\frac{\left(1\wedge \rho^{-4}\right)\left(\rho^2h^2\wedge 1\right)\left(\left|\log \rho\right|+1\right)^{3d}}{\rho^{6H-5}}d\rho\\
= C\bigg(h^2\int_0^1\rho^{7-6H}\left(\left|\log \rho\right|+1\right)^{3d} d\rho
+\int_1^{h^{-1}}h^2\rho^{3-6H}\left(\left|\log \rho\right|+1\right)^{3d}d\rho \\+\int_{h^{-1}}^\infty\rho^{1-6H}\left(\left|\log \rho\right|+1\right)^{3d} d\rho\bigg):=C\left(I_1+I_2+I_3\right).
\end{gather*}
Let us estimate the integrals individually. Obviously, $I_1 \leq C h^2.$
Further, for $\rho\in \left[1,h^{-1}\right]$, $\left(\left|\log \rho\right|+1\right)^{3d}\leq \left(\left|\log h\right|+1\right)^{3d}$, 
whence 
\begin{gather*}
I_2 \leq C \left(\left|\log h\right|+1\right)^{3d}\int_1^{1/h}\rho^{3-6H} d\rho \leq 
C h^{(6H-2)}\left(\left|\log h\right|+1\right)^{3d}.
\end{gather*}
Finally,
\begin{gather*}
I_3=h^{-1}\int_{1}^\infty \left(\frac{\rho}{h}\right)^{1-6H}\left(\log \frac{\rho}{h}+1\right)^{3d}d\rho\\
\leq h^{6H-2}\int_1^\infty \rho^{1-6H}\left(\log \rho+\left|\log h\right|+1\right)^{3d} d\rho\\
\leq h^{6H-2}\left(\left|\log h\right|+1\right)^{3d}\int_1^\infty \rho^{1-6H}\left(\log \rho+1\right)^{3d} d\rho\leq C h^{6H-2}\left(\left|\log h\right|+1\right)^{3d}.
\end{gather*}
Collecting the estimates, we get that
\begin{equation}\label{exp}
\Expect_\xi\left[ Q(h,\xi_k)\right]\leq Ch^{(6H-2)\wedge 2}\left(\left|\log h\right|+1\right)^{3d}.
\end{equation}
Define $b(h)=h^{(6H-2)\wedge 2}\left|\log h\right|^{3d}$, $h \in(0,1/2)$. Taking into account \eqref{e}, \eqref{exp} and the almost sure convergence of the series $\sum_{k=1}^\infty\Gamma_k^{-2/\alpha}$ we get that for any $\varepsilon>0$
$$\Expect_\xi \left[\sum_{n=1}^\infty\frac{a_1\left(2^{-n}\right)}{b\left(2^{-n}\right)n^{1+\varepsilon}}\right]<\infty$$
$\mathsf{P}_\Gamma$-almost surely.
Hence  $\sum_{n=1}^\infty\frac{a_1\left(2^{-n}\right)}{b\left(2^{-n}\right)n^{1+\varepsilon}}<\infty$ \ $\mathsf{P}_\xi\otimes \mathsf{P}_\Gamma$-almost surely. It follows in particular that 
\begin{equation}\label{eq}
\frac{a_1\left(2^{-n}\right)}{b\left(2^{-n}\right)n^{1+\varepsilon}} \rightarrow 0, n\rightarrow \infty,
\end{equation}
$\mathsf{P}_\xi\otimes \mathsf{P}_\Gamma$-almost surely. 
Note that, for positive $h$ small enough, $b(h)$ increases and satisfies $b(2h)\le C b(h)$. Consequently, \eqref{eq} implies that for almost all $(\omega_\xi,\omega_\Gamma) \in\Omega_\xi\times\Omega_\Gamma$
$$a_1(h)\leq C(\omega_\xi,\omega_\Gamma) h^{(6H-2)\wedge 2}\left|\log h\right|^{3d+1+\varepsilon}$$
holds for all sufficiently small $h>0$.

In order to estimate $a_2(h)$, use the following inequality:
\begin{gather*}
\big|\left(1-\cos at'\left|y\right|\right)-\left(1-\cos at''\left|y\right|\right)\big|\leq 2\wedge \big(a\left|y\right||t'-t''|\big).
\end{gather*}
Then 
\begin{gather*}
a_2(h)\leq a^{-2} C_\alpha^2 \sum_{k=1}^\infty\Gamma_k^{-2/\alpha}\varphi(\xi_k)^{-2/\alpha}\frac{4 \wedge \big(a^2\left|\xi_k\right|^2h^2\big)}{\left|\xi_k\right|^4}\\
\times\big(1\wedge (|x|^2|\xi_k|^2)\big)\prod_{l=1}^3\frac{1}{\left|\xi_{k,l}\right|^{2H+2/\alpha-2}}.
\end{gather*}
Using the same reasoning as above, we get that for almost all $(\omega_\xi,\omega_\Gamma) \in \Omega_\xi\times\Omega_\Gamma$
$$a_2(h)\leq C (\omega_\xi,\omega_\Gamma)h^{(6H-2)\wedge 2}\left|\log h\right|^{3d+1+\varepsilon}.$$
As a result, we get for almost all $\omega_\xi \in \Omega_\xi,\omega_\Gamma \in \Omega_\Gamma$
\begin{gather*}
\sup_{A_{R,h}}\Expect_g\left[\left|U(x',t')-U(x'',t'')\right|^2\right]\leq C(\omega_\xi,\omega_\Gamma)h^{(6H-2)\wedge 2}\left|\log h\right|^{3d+1+\varepsilon},
\end{gather*}
for all $h$ small enough, 
where $$A_{R,h} = \big\{(x',x'',t',t'')\in\R^6\times[0,T]^2: |x'|, |x''|\leq R,\
|x'-x''|\le h,
|t'-t''|\le h\big\}.$$
For a fixed $(\omega_\xi,\omega_\Gamma) \in \Omega_\xi\times\Omega_\Gamma$,  $U$ has a centered Gaussian distribution. Therefore (see e.g.\ \cite[p.~220]{lifshits}) there is a modification of $U$ satisfying
\begin{gather*}
\sup_{A_{R,h}} 
\left|U(x',t')-U(x'',t'')\right|\leq C(\omega)h^{(3H-1)\wedge 1}\left|\log h\right|^{3d/2+1+\varepsilon/2}.
\end{gather*}
Recalling that $d = (1+\eta)(2/\alpha -1)$, we get the required statement by choosing $\varepsilon$ and $\eta$ sufficiently small.

2. Assuming that the continuous modification of $U$ is chosen, define 
$$g(x,y,t)=\frac{\partial}{\partial t}A^Hf_{x,t}=\sin\left(x,y\right)\frac{\sin at\left|y\right|}{\left|y\right|^2}\left|y\right|\,\prod_{l=1}^3\frac{\sign y_{l}}{\left|y_{l}\right|^{H+1/\alpha-1}}$$
and 
$$V(x,t)=\myiiint_{\R^3}g(x,y,t)M(dy)=\myiiint_{\R^3}\sin\left(x,y\right)\frac{\sin at\left|y\right|}{\left|y\right|}\, \prod_{l=1}^3\frac{\sign y_{l}}{\left|y_{l}\right|^{H+1/\alpha-1}}M(dy).$$
With the help of spherical change of variables \eqref{sfz}, estimate
\begin{gather*}
\myiiint_{\R^3}\left|g(x,y,t)\right|^\alpha dy\leq \myiiint_{\R^3}\left(\frac{\left(1\wedge \left|x\right|\left|y\right|\right)\left(1\wedge \left|a\right|\left|t\right|\left|y\right|\right)}{|y| \prod_{l=1}^3\left|y_{l}\right|^{H+1/\alpha-1}\left|y\right|}\right)^\alpha dy\\
=\int_0^\infty\int_0^\pi\int_0^{2\pi}\frac{\left(1\wedge \left|x\right|\rho\right)^\alpha\left(1\wedge \left|a\right|\left|t\right|\rho\right)^\alpha\rho^2\sin\theta}{\rho^{3\alpha H+3-2\alpha}\left|\sin\theta\right|^{2\alpha H+2-2\alpha}}\\
\times\left(\cos\varphi\sin\varphi\cos\theta\right)^{\alpha-1-\alpha H}d\theta\, d\varphi\, d\rho\\
=\int_0^\infty\int_0^\pi\int_0^{2\pi}\left(1\wedge \left|x\right|^\alpha\rho^\alpha\right)\left(1\wedge \left|a\right|^\alpha\left|t\right|^\alpha\rho^\alpha\right)
\\
\times\left(\cos\varphi\sin\varphi\cos\theta\right)^{\alpha-1-\alpha H}\left(\sin\theta\right)^{2\alpha-2\alpha H-1}\rho^{2\alpha-3\alpha H-1}d\theta\, d\varphi\, d\rho\\
\leq C \int_0^\infty \left(1\wedge \left|x\right|^\alpha\rho^\alpha\right)\left(1\wedge \left|a\right|^\alpha\left|t\right|^\alpha\rho^\alpha\right)\rho^{2\alpha-3\alpha H-1}d\rho<\infty;
\end{gather*}
the integral is convergent since $4\alpha-3\alpha H -1>-1$ and  $2\alpha-3\alpha H-1<-1$ for $\alpha\in (1,2)$ and  $H\in ({2}/{3},1)$.
As a result,
$$\int_0^t\left(\myiiint_{\R^3}\left|g(x,y,s)\right|^\alpha dy\right)^{1/\alpha}ds<\infty,$$
so by Theorem \ref{tv32t41}
$$\int_0^t V(x,s)ds=\myiiint_{\R^3}\int_0^tg(x,y,s)ds\, M(dy)=\myiiint_{\R^3}A^H f_{x,t}(y)M(dy)=U(x,t)$$
almost surely. Since both sides of the last equality are continuous, and the left-hand side is absolutely continuous, then we get that, almost surely, $U$ is absolutely continuous in $t$. The absolute continuity in other variables is shown similarly. 
\end{proof}

\appendix
\section{Definition and properties of  $A^H$}

Let $g\in L^1(\R^3)\cap L^2(\R^3)$. Define
\begin{equation}\label{ah}
A^{H}g(y)=\Real\widehat{g}(y)\prod_{l=1}^3\frac{\sign y_l }{\left|y_l\right|^{H+1/\alpha-1}},
\end{equation}
where $\widehat{g}(y)=\myiiint_{\R^3}e^{i(x,y)}g(x)dx$ is the Fourier transform of $g$.
\begin{tv}\label{a1}
For any $\alpha \in (0,2)$, $H\in(1/2,1)$, there exists a constant $C(H,\alpha)$ such that for any $g \in L^1(\R^3)\cap L^2(\R^3)$, 
$$\left\|A^{H}g\right\|_{L^\alpha(\R^3)}\leq C(H,\alpha) \bigl(\left\|g\right\|_{L^1(\R^3)} + \left\|g\right\|_{L^2(\R^3)}\bigr).$$
\end{tv}
\begin{proof}
We need to estimate $\myiiint_{\R^3}\left|A^{H}g(y)\right|^\alpha dy$. Let us split the integral into two: over the unit ball $B(0,1)$ and over its complement $B(0,1)^c$. To estimate the former, write 
\begin{gather*}
\myiiint_{B(0,1)}\left|A^{H}g(y)\right|^\alpha dy =\myiiint_{ B(0,1)}\left|\Real \widehat{g}(y)\right|^\alpha \prod_{l=1}^3\frac{1}{\left|y_l\right|^{\alpha(H-1)+1}}dy_l\\
\leq \|g\|^\alpha_{L^1(\R^3)}\myiiint_{ B(0,1)}\prod_{l=1}^3\frac{1}{\left|y_l\right|^{\alpha(H-1)+1}}dy_l,
\end{gather*}
and the integral is finite, since $H<1$. 

To estimate the integral over $B(0,1)^c$, use the H\"older inequality with $p=2/\alpha, q = 2/(2-\alpha)$ to get
\begin{equation}\label{B}
\begin{gathered}
\myiiint_{B(0,1)^c}\left|A^{H}g(y)\right|^\alpha dy =\myiiint_{ B(0,1)^c}\left|\Real \widehat{g}(y)\right|^\alpha \prod_{l=1}^3\frac{1}{\left|y_l\right|^{\alpha(H-1)+1}}dy_l\\
\leq \left(\myiiint_{ B(0,1)^c}\left|\Real \widehat{g}(y)\right|^2dy\right)^{\alpha/2}\left(\myiiint_{ B(0,1)^c}\prod_{l=1}^3\left|y_l\right|^{\alpha q(1-H)-q}dy_l\right)^{1/q}.
\end{gathered}
\end{equation}
In the second integral make the spherical change of variables \eqref{sfz} to obtain
\begin{gather*}
\myiiint_{B(0,1)^c}\prod_{l=1}^3\left|y_l\right|^{\alpha q(1-H)-q}dy_l\leq C \int_1^\infty \rho^{3\alpha q(1-H)-3q+2}d\rho <\infty,
\end{gather*}
since the inequality $3\alpha q(1-H)-3q+2<-1$ easily transforms to $\alpha (1-2H)<0$. For the first integral in \eqref{B} we use the Parseval identity: 
\begin{gather*}
\myiiint_{ B(0,1)^c}\left|\Real \widehat{g}(y)\right|^2dy \le \myiiint_{\R^3}\left|\widehat{g}(y)\right|^2dy=\frac{1}{\left(2\pi\right)^3}\myiiint_{\R^3}\left(g(y)\right)^2dy.
\end{gather*}
This implies the required estimate.
\end{proof}

\bibliographystyle{degruyter-plain}
\bibliography{pryhara-wc}

\end{document}